\documentclass[12pt]{article}
\usepackage[usenames]{color}
\usepackage{graphicx}
\usepackage[mathcal]{euscript}
\usepackage{stmaryrd}

\usepackage{amsmath,amssymb,amsthm,stmaryrd}
\usepackage[all]{xy}
\xyoption{arc}

\newtheorem{thm}{Theorem}
\newtheorem{cor}[thm]{Corollary}
\newtheorem{prop}[thm]{Proposition}
\newtheorem{lem}[thm]{Lemma}
\theoremstyle{definition}
\newtheorem{defn}[thm]{Definition}
\theoremstyle{remark}

\def\co{\colon\thinspace}
\newcommand{\mb}[1]{\mathbb{#1}}
\newcommand{\mf}[1]{\mathfrak{#1}}

\newcommand{\Hom}{\ensuremath{{\rm Hom}}}

\newcommand{\Map}{\ensuremath{{\rm Map}}}

\newcommand{\Ss}{\textbf{S}}
\newcommand{\Cat}{{\mathcal{C}\mathrm{at}}}
\newcommand{\Set}{{\mathcal{S}\mathrm{et}}}
\newcommand{\op}{{op}}

\title{Localization of enriched categories and cubical sets}
\author{Tyler Lawson\thanks{Partially supported by NSF
    grant DMS--1206008.}}

\begin{document}
\maketitle

\begin{abstract}
  The invertibility hypothesis for a monoidal model category $\Ss$
  asks that localizing an $\Ss$-enriched category with respect to an
  equivalence results in an weakly equivalent enriched category. This
  is the most technical among the axioms for $\Ss$ to be an {\em
    excellent model category} in the sense of Lurie, who showed that
  the category $\Cat_\Ss$ of $\Ss$-enriched categories then has a model
  structure with characterizable fibrant objects. We use a universal
  property of cubical sets, as a monoidal model category, to show that
  the invertibility hypothesis is consequence of the other axioms.
\end{abstract}

Topological categories, simplicial categories, and differential graded
categories are special types of enriched categories: the enriching
category has a notion of weak equivalence and its own homotopy
theory. These have played a prominent role a diverse array of
subjects.

Getting control over the homotopy theory of some of these enriched
categories and homotopical constructions in them (such as pushouts,
pullbacks, and other derived limit and colimit constructions) is
easier in the presence of model structures. If $\Ss$ is a monoidal
model category, Lurie gave conditions for the existence of a model
structure with many useful properties on the collection $\Cat_\Ss$ of
$\Ss$-enriched categories \cite[A.3.2.4]{lurie-htt}. The cofibrations
and weak equivalences in $\Cat_\Ss$ have a relatively straightforward
description (see \S\ref{sec:enriched-categories}), but in order to get
a useful characterization of the fibrations more assumptions are
required. With this goal, Lurie defined an {\em excellent} model
category as a model category $\Ss$, with a symmetric monoidal
structure, satisfying additional axioms labeled (A1) through (A5). The
first four of these axioms are all relatively standard concepts or are
straightforward to verify.

Axiom (A5) is called the invertibility hypothesis. It is more
technical---it roughly asserts that inverting a weak equivalence
results in a weakly equivalent enriched category---and is more
difficult to verify in practice.  The fact that the category
$\Set_\Delta$ of simplicial sets satisfies the invertibility
hypothesis is one of the main results in
\cite{dwyer-kan-simpliciallocalization}. The invertibility hypothesis
for differential graded categories is a consequence of
\cite[8.7]{toen-derivedmorita}, and for enrichment in simplicial model
categories it is one of the main theorems of
\cite{dundas-localization}.

Our main result is the following.

\begin{thm}
\label{thm:main}
  Let $\Ss$ be a combinatorial monoidal model category. Assume that
  every object of $\Ss$ is cofibrant and that the collection of weak
  equivalences in $\Ss$ is stable under filtered colimits. Then $\Ss$
  satisfies the invertibility hypothesis of \cite[A.3.2.12]{lurie-htt}.
\end{thm}

Combinatoriality is Lurie's axiom (A1), cofibrance of all objects is a
consequence of axiom (A2), stability of weak equivalences under
filtered colimits is axiom (A3), and the model structure being monoidal
is axiom (A4). (Lurie also asks as part of the definition that the
model category be symmetric monoidal.)

Our method is the following. We will first show that the category
$\Set_\square$ of cubical sets, with a model structure due to Cisinski
\cite{cisinski-testcat}, admits a monoidal left Quillen functor out to
essentially {\em any} monoidal model category $\Ss$: the choice of
such a functor is essentially a choice of cylinder object for the
monoidal unit. Second, we will show (in a method adapted from
{\cite[A.3.2.20, A.3.2.21]{lurie-htt}}) that this left Quillen functor
allows the model category $\Ss$ to inherit the invertibility
hypothesis from $\Set_\square$.

This gives cubical sets a useful universal property. We also note that
the construction $\mf C[-]$ of \cite[\S 1.1.5]{lurie-htt} which takes
simplicial sets to categories enriched in simplicial sets, fundamental
to the study of quasicategories, factors naturally through categories
enriched in cubical sets. However, some good properties of simplicial
sets are lost: the monoidal structure on cubical sets is not
symmetric, and understanding the homotopy theory of cubical sets
requires hard theorems.

The author would like to thank Adeel Khan for discussion related to
this paper.

\section{Cubical sets}

We'll begin with preliminaries on the category of cubical sets.

\begin{itemize}
\item The cube category $\square$ has, as objects, the $n$-cubes
  $\square^n$ for $n \geq 0$.
\item The maps $\square^n \to \square^m$ are in bijective
  correspondence with the set of maps $[0,1]^n \to [0,1]^m$ which are
  composites of coordinate projections $(x_1,\dots,x_n) \mapsto
  (x_1,\dots,\widehat{x_i},\dots,x_n)$ and face inclusions
  $(x_1,\dots,x_n) \mapsto (x_1,\dots,x_{i-1},\epsilon,x_i,\dots,x_n)$ 
  for $\epsilon \in \{0,1\}$.
\item There is a monoidal product $\otimes\co \square \times \square
  \to \square$ such that $\square^n \otimes \square^m =
  \square^{n+m}$, corresponding to the isomorphism $[0,1]^n \times
  [0,1]^m \cong [0,1]^{n+m}$. (This monoidal product is not
  symmetric.)
\end{itemize}

Write $\square^{\leq 1}$ for the full subcategory of $\square$ spanned
by $\square^0$ and $\square^1$. In this subcategory, there are two
maps $j^0, j^1\co \square^0 \to \square^1$ and a map $r\co \square^1
\to \square^0$, and these satisfy $r j^0 = r j^1 = id_{\square^0}$;
moreover, these maps and these relations generate all maps and
relations in $\square^{\leq 1}$.

The category $\square$ is generated by $\square^{\leq 1}$ and the
monoidal structure in the following sense.
\begin{prop}
  Suppose $\mathcal{C}$ is a monoidal category. For any functor
  $F\co \square^{\leq 1} \to \mathcal{C}$ such that $F(\square^0)$ is
  the monoidal unit $\mb I$, there exists an extension to a
  functor $\tilde F\co \square \to \mathcal{C}$ which is
  monoidal. This extension is unique up to natural isomorphism.
\end{prop}

A cubical set is a functor $\square^\op \to \Set$, and $\Set_\square$
is the category of cubical sets. The covariant Yoneda embedding
$\theta\co \square \to \Set_\square$ satisfies the following
property.
\begin{prop}
  Suppose $\mathcal{C}$ is cocomplete. For any functor $F\co \square
  \to \mathcal{C}$, the ``singular cubical set'' functor
  $\Hom_{\mathcal{C}}(F(\square^\bullet), -)$ has a left adjoint
  $\tilde F\co \Set_\square \to \mathcal{C}$ extending $F$. This
  extension is unique up to natural isomorphism.
\end{prop}

The monoidal structure on $\square$ gives rise to a Day convolution
product $\otimes$ on $\Set_\square$. More specifically, given $X,Y\co
\square^\op \to \Set$, the tensor $X \otimes Y$ is the left Kan
extension of $X \times Y\co \square^\op \times \square^\op \to \Set
\times \Set \to \Set$ along the functor $\otimes\co \square^\op \times
\square^\op \to \square^\op$. The universal property of left Kan
extension gives this (nonsymmetric) monoidal product a universal
property as well.
\begin{prop}
  Suppose $\mathcal{C}$ is a cocomplete category with a monoidal
  structure that preserves colimits in each variable separately. For
  any monoidal functor $F\co \square \to \mathcal{C}$, the
  colimit-preserving extension $\tilde F\co \Set_\square \to
  \mathcal{C}$ is also monoidal.
\end{prop}

We now consider monoidal model categories. Recall that for maps
$f_1\co A_1 \to B_1$ and $f_2\co A_2 \to B_2$ in a cocomplete monoidal
category $\mathcal{C}$, the pushout-product $f_1 \boxtimes f_2$ is the
map
\[
A_1 \otimes B_2 \coprod_{A_1 \otimes A_2} B_1 \otimes A_2 \to B_1
\otimes B_2.
\]
Write $i$ for the map $j^0 \amalg j^1\co \square^0 \coprod
\square^0 \to \square^1$ in $\Set_\square$. The pushout-product
allows us to define cubical sets $\partial \square^n$ (for $n \geq 0$)
and $\sqcap^n_{(k,\epsilon)}$ (for $1 \leq k \leq n$ and $\epsilon \in
\{0,1\})$ as the sources of the following pushout-product maps:
\begin{align}
\label{eq:gencof}
(i \boxtimes \dots \boxtimes i) &\co \partial \square^n \to
\square^n\\
\label{eq:genaccof}
(i \boxtimes \dots \boxtimes j^\epsilon \boxtimes \dots \boxtimes i) &\co
\sqcap^n_{(k,\epsilon)} \to \square^n
\end{align}
(We will use the convention that, when $n=0$, the map $\partial
\square^0 \to \square^0$ of Equation~(\ref{eq:gencof}) is the map
$\emptyset \to \square^0$ from the initial object.)

We will require the following result of Cisinski.
\begin{thm}[{\cite{cisinski-testcat, jardine-categorical}}]
  There exists a combinatorial, left proper model structure on
  $\Set_\square$ with generating cofibrations the maps of
  Equation~\eqref{eq:gencof} and generating acyclic cofibrations the
  maps of Equation~\eqref{eq:genaccof}. There is a monoidal left
  Quillen equivalence $\Set_\square \to \Set_\Delta$, given by the
  cubical realization functor which sends $\square^n$ to
  $(\Delta^1)^n$.
\end{thm}

\begin{cor}
  \label{cor:fromcubical}
  Suppose $\mathcal{C}$ is a monoidal model category in the sense of
  \cite[\S 4]{hovey-modelcategories}. Let $F\co \square^{\leq 1} \to
  \mathcal{C}$ be a functor, and let $\tilde F\co \Set_\square \to
  \mathcal{C}$ be an extension to a monoidal left adjoint. Then
  $\tilde F$ is a left Quillen functor if and only if the unit of
  $\mathcal{C}$ is cofibrant and the maps $F(j^0)$, $F(j^1)$ express
  $F(\square^1)$ as a cylinder object for the unit of
  $\mathcal{C}$. In particular, such a monoidal left Quillen functor
  exists.
\end{cor}

\begin{proof}
  In order for $\tilde F$ to be a left Quillen functor, it must take
  the acyclic cofibrations $j^0$ and $j^1$ to acyclic cofibrations and
  the cofibration $i$ to a cofibration; this happens precisely when
  $F(\square^1)$ is expressed as a cylinder object. It also must take
  the map $\emptyset \to \square^0$ to a cofibration, so the unit must
  be cofibrant.

  Conversely, suppose that $F$ expresses $F(\Delta^1)$ as a cylinder
  object, so that the map $\tilde F(i)$ is a cofibration and that the
  maps $\tilde F(j^\epsilon)$ are both acyclic cofibrations. Then the
  fact that $F$ is monoidal and colimit-preserving implies that
  $\tilde F(f_1 \boxtimes \dots \boxtimes f_n)$ is the pushout-product
  $\tilde F(f_1) \boxtimes \dots \boxtimes \tilde F(f_n)$. This makes
  the map $\tilde F(i\boxtimes \dots \boxtimes i)$ into an iterated
  pushout-product of cofibrations, and makes the map $F(i\boxtimes
  \dots\boxtimes j^\epsilon \boxtimes \dots \boxtimes i)$ into an
  iterated pushout-product of several cofibrations and one acyclic
  cofibration. As $\mathcal{C}$ is a monoidal model category with
  cofibrant unit, $\tilde F$ then preserves the generating
  cofibrations and generating acyclic cofibrations, and hence is a
  left Quillen functor.
\end{proof}

\section{Enriched categories}
\label{sec:enriched-categories}

Suppose that $\Ss$ is a monoidal model category with unit $\mb I$,
that every object of $\Ss$ is cofibrant, and that the collection of
weak equivalences in $\Ss$ is stable under filtered colimits. For such
$\Ss$, Lurie constructs a left proper combinatorial model structure on
$\Cat_\Ss$ in \cite[A.3.2.4]{lurie-htt} which we will review now.

Following the notation of \cite[\S A.3.2]{lurie-htt}, we define the
following four special examples of enriched categories:
\begin{itemize}
\item Let $\emptyset$ be the trivial enriched category with no objects.
\item Let $[0]_\Ss$ be the category with a single object $0$ and
  $\Map_\mathcal{C}(0,0) = \mb I$.
\item Let $[1]_A$ be the category with objects $0$ and $1$, such that
  $\Map_{[1]_A}(0,1) = A$, $\Map_{[1]_A}(i,i) = \mb I$, and
  $\Map_{[1]_A}(1,0) = \emptyset$. If $A = \mb I$, we simply write
  $[1]_\Ss$ for $[1]_{\mb I}$.
\item Let $[1]_{\tilde \Ss}$ be the category with objects $0$ and
  $1$, such that $\Map_{[1]_{\tilde \Ss}}(i,j) = \mb I$ for all $i$ and $j$.
\end{itemize}

The model structure on $\Cat_\Ss$ is defined by the following
requirements:
\begin{itemize}
\item An enriched functor $F\co \mathcal{C} \to \mathcal{D}$ is a weak
  equivalence if the map $h\mathcal{C} \to h\mathcal{D}$ of homotopy
  categories, obtained by applying $[\mb I,-]_{h\Ss}$ to morphism
  objects, is an equivalence, and if for all $c,c' \in \mathcal{C}$
  the map $\Map_{\mathcal{C}}(c,c') \to \Map_{\mathcal{D}}(Fc,Fc')$ is
  an equivalence in $\Ss$.
\item The set
\[
\Big\{[1]_S \to [1]_{S'}
  \mid S \to S' \text{ a generating cofibration}\Big\} \cup \Big\{\emptyset \to [0]_\Ss\Big\}
\]
is a set of generating cofibrations.
\end{itemize}
As a consequence, a monoidal left Quillen functor $F\co \Ss \to \Ss'$
between such categories gives rise to a left Quillen functor $\Cat_\Ss
\to \Cat_{\Ss'}$, which is a left Quillen equivalence if $F$ was
\cite[A.3.2.6]{lurie-htt}.

In the following, we will write $\Cat_\square$ for the category
$\Cat_{\Set_\square}$ of categories enriched in cubical sets, and
similarly $\Cat_\Delta$ for the category $\Cat_{\Set_\Delta}$ of
categories enriched in simplicial sets.

\begin{defn}
  Let $\mathcal{C} \in \Cat_\square$ be a category enriched in cubical
  sets. Given morphisms $f,g\co c \to c'$ in the underlying
  category $\mathcal{C}$, classified by maps $f,g\co \square^0 \to
  \Map_{\mathcal{C}}(c,c')$, a homotopy from $f$ to $g$
  identity is a morphism $H\co \square^1 \to \Map_{\mathcal{C}}(c,c')$
  such that $H \circ j^0 = f$ and $H \circ j^1 = g$.

  The {\em homotopy inverse} category $\mathcal{H} \in
  \Cat_\square$ is universal among cubically enriched categories
  possessing morphisms $u\co c \to c'$ and $v\co c' \to c$ together
  with a homotopy from $v \cdot u$ to the identity $id_c$.
\end{defn}

The category $\mathcal{H}$ can be described as an iterated pushout
diagram in $\Cat_\square$ as follows.
\[
\xymatrix{
[1]_\emptyset \ar[d] \ar[r] & [1]_{\square} \ar[d]^u &&
[1]_{\partial \square^1} \ar[r] \ar[d] & P \ar[d]\\
[1]_{\square} \ar[r]_v & P &&
[1]_{\square^1} \ar[r]_H & \mathcal{H}
}
\]
The category $P$ classifies a pair of morphisms $u$ and $v$ in
opposing directions, and the pushout defining $\mathcal{H}$ adjoins
a homotopy from $v \cdot u$ to $id_c$. This definition allows us to
deduce the following properties.
\begin{prop}
  The maps $\ell\co [1]_\square \to \mathcal{H}$ and $r\co
  [1]_\square \to \mathcal{H}$, classifying the functions $u$ and $v$
  respectively, are cofibrations. A map of enriched categories
  $[1]_\square \to \mathcal{C}$, classifying a map $f\co c \to c'$ in
  $\mathcal{C}$, has an extension along $\ell$ if and only if $f$
  admits a left homotopy inverse, and has an extension along $r$ if
  and only if $g$ admits a right homotopy inverse.
\end{prop}

\begin{defn}
  We define the category $\mathcal{E} \in \Cat_\square$ as the pushout
  in the diagram
  \[
  \xymatrix{
    [1]_\square \ar[r]^\ell \ar[d]_r & \mathcal{H} \ar[d] \\
    \mathcal{H} \ar[r] & \mathcal{E}.
  }
  \]
\end{defn}

The map $[1]_\square \to \mathcal{E}$ is a cofibration.  More, since
the map $[1]_\square \to [1]_{\tilde \square}$ sends the morphism $f$
to one with an inverse, there is a factorization $[1]_\square \to
\mathcal{E} \to [1]_{\tilde \square}$ (which, as the target category
is discrete, is unique).

\begin{lem}
  \label{lem:invextension}
  Suppose $\mathcal{C} \in \Cat_\square$ is fibrant, and $f\co
  [1]_\square \to \mathcal{C}$ classifies a map which becomes an
  isomorphism in the homotopy category $h\mathcal{C}$. Then $f$
  extends to a map $\mathcal{E} \to \mathcal{C}$.
\end{lem}

\begin{proof}
  We need to show that $f$ extends along the maps $\ell, r \co
  [1]_\square \to \mathcal{H}$. First assume that $f$ has a left
  inverse $g$: we will show that it extends over $\ell$.

  We note that since $\mathcal{C}$ is a fibrant
  $\Set_\square$-enriched category, the objects
  $\Map_{\mathcal{C}}(c,c')$ are fibrant for all $c$ and $c'$ (cf. the
  proof of \cite[A.3.2.24]{lurie-htt}). Therefore, since $f\co
  \square^0 \to \Map_{\mathcal{C}}(c,c')$ has a left inverse in the
  homotopy category, the left inverse has a representative $g$ in the
  form of a map $g\co \square^0 \to \Map_{\mathcal{C}}(c',c)$. More,
  since the composite $g \cdot f$ becomes equal to $id_c$ as maps
  $\square^0 \to \Map_{\mathcal{C}}(c,c)$ in the homotopy category of
  $\Ss$, these maps are left homotopic. In particular, since
  $\square^1$ is a cylinder object for $\square^0$ in $\Set_\square$,
  there exists an extension in the following diagram:
\[
\xymatrix{
  \partial \square^1 \ar[rr]^-{g\cdot f \coprod id_c} \ar[d] &&
  \Map_{\mathcal{C}}(c,c).\\
  \square^1 \ar@{.>}[urr]_H
}
\]
  The description of the map $\ell\co [1]_\square \to \mathcal{H}$ by
  its universal property then produces precisely the desired
  extension.

  The extension over $r$ is symmetric. (Note that that it is the proof
  itself that is symmetric: we cannot, for example, proceed by taking
  the opposite category of $\mathcal{H}$ because cubically enriched
  categories do not have a natural opposite category.)
\end{proof}

We now consider what happens when the map $f$ is inverted.

\begin{prop}
\label{prop:localization}
  Let $\mathcal{E}\langle f^{-1}\rangle$ be the localization of
  $\mathcal{E}$ obtained by inverting the map $f$. Then in the
  diagram
  \[
  \mathcal{E} \to \mathcal{E}\langle f^{-1}\rangle \to   [1]_{\tilde
    \square},
  \]
  which is determined by the universal property of the localization,
  both maps are weak equivalences in $\Cat_\square$.
\end{prop}

\begin{proof}
  Because there is a monoidal left Quillen equivalence $L\co
  \Cat_\square \to \Cat_\Delta$, which preserves pushouts and
  localizations, it suffices to show that these maps of cubically
  enriched categories becomes equivalences of simplicially enriched
  categories.
  
  The category $L(\mathcal{E})$ is the universal simplicial category
  with a morphism $f$ together with a homotopy left inverse and a
  homotopy right inverse; moreover, the map $[1]_{\Delta} \to
  L(\mathcal{E})$ classifying $f$ is a cofibration (as the image of a
  cofibration under $L$).

  We may then apply work of Dwyer and Kan
  \cite{dwyer-kan-simpliciallocalization}, which shows that the
  localization map $L(\mathcal{E}) \to L(\mathcal{E})\langle
  f^{-1}\rangle$ that inverts $f$ is a weak equivalence of
  simplicially enriched categories.

  We now need to show that the map $L(\mathcal{E})\langle
  f^{-1}\rangle \to [1]_{\tilde \Delta}$ is an equivalence. This
  localization is still an iterated pushout, but as its two objects
  are now isomorphic it may be reinterpreted: it is the universal
  example of a simplicial category with two objects $c$ and $c'$, an
  isomorphism $c \to c'$, two maps $g_1, g_2\co c \to c$, and two
  homotopies $H_1\co g_1 \simeq id_c$ and $H_2\co g_2 \simeq id_c$.
  To show that this is equivalent to $[1]_{\tilde \Delta}$, we must
  show that the mapping spaces in $L(\mathcal{E})$ are all weakly
  equivalent to $\Delta^0$.

  The full subcategory $\mathcal{E}' \subset L(\mathcal{E})\langle
  f^{-1} \rangle$ spanned by the object $c$ is weakly equivalent to
  this one. As it has one object, $\mathcal{E}'$ is determined
  completely by the simplicial monoid $\Map(c,c)$, which is the free
  simplicial monoid with two elements $g_i$ and paths from $g_i$ to
  the identity. This monoid is the James construction $J(\Delta^1 \vee
  \Delta^1)$ on the based simplicial set $\Delta^1 \vee \Delta^1$, and
  as such it is weakly equivalent to $\Delta^0$: its realization
  $J([0,1] \vee [0,1])$ as a topological space is contractible.
\end{proof}

We note that for this result to hold, it is important that the
construction of ${\cal E}$ not ask for the left and right inverses
of $f$ to be the same map: in the final step we would instead obtain
$J(S^1) \simeq \Omega S^2$ rather than a contractible space if we did
so.

\begin{cor}
  \label{cor:cofreplacement}
  The map $[1]_\square \to \mathcal{E}$ is a cofibrant replacement for
  $[1]_\square \to [1]_{\tilde \square}$ in $\Cat_\square$.
\end{cor}

\section{The invertibility hypothesis}
\label{sec:invert-hypoth}

We now recall the precise statement of the invertibility hypothesis
\cite[A.3.2.12]{lurie-htt}.
\begin{defn}
  \label{def:invhyp}
  Let $\Ss$ be a combinatorial monoidal model category. Assume that
  every object of $\Ss$ is cofibrant and that the collection of weak
  equivalences in $\Ss$ is stable under filtered colimits. We say that
  $\Ss$ satisfies the invertibility hypothesis when, for any
  isomorphism in the homotopy category $h\mathcal{C}$ classified by a
  cofibration $f\co [1]_\Ss \to \mathcal{C}$, if we form the pushout
  diagram of $\Ss$-enriched categories
\[
\xymatrix{
[1]_\Ss \ar[r]^i \ar[d] & \mathcal{C} \ar[d]^j \\
[1]_{\tilde \Ss} \ar[r] & \mathcal{C}\langle f^{-1} \rangle,
}
\]
the map $j$ is a weak equivalence.
\end{defn}

We will now prove our main result.

\begin{proof}[Proof of Theorem~\ref{thm:main}]
  Let $f\co [1]_\Ss \to \mathcal{C}$ be as in
  Definition~\ref{def:invhyp}.  As noted in
  \cite[A.3.2.13]{lurie-htt}, it suffices to verify that this
  condition is satisfied in the case where $\mathcal{C}$ is a fibrant
  $\Ss$-enriched category: in particular, we may assume that the
  mapping objects in $\mathcal{C}$ are fibrant.

  By Corollary~\ref{cor:fromcubical}, there exists a monoidal left
  Quillen functor $L\co \Set_\square \to \Ss$ from the
  category of cubical sets, sending $\square^1$ to a cylinder object
  for the monoidal unit. This induces a left Quillen functor $L\co
  \Cat_\square \to \Cat_\Ss$ with right adjoint $R$. The right adjoint
  preserves homotopy categories:
  \begin{align*}
    \Hom_{h\mathcal{C}}(c,c')
    &= \Hom_{h\Ss}([1]_\Ss,\Map_{\mathcal{C}}(c,c'))\\
    &= \Hom_{h\Ss}(L[1]_\square,\Map_{\mathcal{C}}(c,c'))\\
    &\cong \Hom_{h\Set_\square}([1]_\square,R\Map_{\mathcal{C}}(c,c'))\\
    &= \Hom_{hR\mathcal{C}}(c,c')
  \end{align*}
  The map $f$ becomes an isomorphism in the homotopy category of
  $\mathcal{C}$, so Lemma~\ref{lem:invextension} implies that the map
  $[1]_\square \to R\mathcal{C}$ classifying $f$ extends to a map
  $\mathcal{E} \to R\mathcal{C}$. The adjoint is an extension $[1]_\Ss
  \to L(\mathcal{E}) \to \mathcal{C}$.

  Now factor $L(\mathcal{E}) \to \mathcal{C}$ into a cofibration
  $L(\mathcal{E}) \to \mathcal{C}'$ followed by an acyclic fibration
  $\mathcal{C}' \to \mathcal{C}$.  The pushouts of $\mathcal{C}$ and
  $\mathcal{C}'$ along $[1]_\Ss \to [1]_{\tilde \Ss}$ are equivalent
  and so we may assume without loss of generality that $L(\mathcal{E})
  \to \mathcal{C}$ is a cofibration.

  Consider the following diagram of pushouts.
\[
\xymatrix{
[1]_\Ss \ar[r] \ar[d] &
L(\mathcal{E}) \ar[r] \ar[d] &
\mathcal{C} \ar[d] \\
% L(\mathcal{E}) \ar[r] \ar[d] &
% L(\mathcal{E} \coprod_{[1]_\square} \mathcal{E}) \ar[r] \ar[d] &
% \mathcal{C}[f^{-1}] \ar[d] \\
[1]_{\tilde \Ss} \ar[r] &
L(\mathcal{E}\langle f^{-1} \rangle) \ar[r] &
\mathcal{C}\langle f^{-1}\rangle
}
\]
  The top row consists of cofibrations. Lemma~\ref{lem:invextension}
  implies that the center vertical map is an equivalence. As the model
  structure on $\Cat_\Ss$ is left proper, the right-hand vertical map
  is also an equivalence.\end{proof}

\bibliography{../masterbib}

\providecommand{\bysame}{\leavevmode\hbox to3em{\hrulefill}\thinspace}
\providecommand{\MR}{\relax\ifhmode\unskip\space\fi MR }
\providecommand{\MRhref}[2]{%
  \href{http://www.ams.org/mathscinet-getitem?mr=#1}{#2}
}
\providecommand{\href}[2]{#2}
\begin{thebibliography}{Dun01}

\bibitem[Cis06]{cisinski-testcat}
Denis-Charles Cisinski, \emph{Les pr\'efaisceaux comme mod\`eles des types
  d'homotopie}, Ast\'erisque (2006), no.~308, xxiv+390. \MR{2294028
  (2007k:55002)}

\bibitem[DK80]{dwyer-kan-simpliciallocalization}
W.~G. Dwyer and D.~M. Kan, \emph{Simplicial localizations of categories}, J.
  Pure Appl. Algebra \textbf{17} (1980), no.~3, 267--284. \MR{579087
  (81h:55018)}

\bibitem[Dun01]{dundas-localization}
Bj{\o}rn~Ian Dundas, \emph{Localization of {$V$}-categories}, Theory Appl.
  Categ. \textbf{8} (2001), No.\ 10, 284--312. \MR{1835445 (2002b:18009)}

\bibitem[Hov99]{hovey-modelcategories}
Mark Hovey, \emph{Model categories}, Mathematical Surveys and Monographs,
  vol.~63, American Mathematical Society, Providence, RI, 1999. \MR{1650134
  (99h:55031)}

\bibitem[Jar06]{jardine-categorical}
J.~F. Jardine, \emph{Categorical homotopy theory}, Homology, Homotopy Appl.
  \textbf{8} (2006), no.~1, 71--144. \MR{2205215 (2006j:55010)}

\bibitem[Lur09]{lurie-htt}
Jacob Lurie, \emph{Higher topos theory}, Annals of Mathematics Studies, vol.
  170, Princeton University Press, Princeton, NJ, 2009. \MR{2522659
  (2010j:18001)}

\bibitem[To{\"e}07]{toen-derivedmorita}
Bertrand To{\"e}n, \emph{The homotopy theory of {$dg$}-categories and derived
  {M}orita theory}, Invent. Math. \textbf{167} (2007), no.~3, 615--667.
  \MR{2276263 (2008a:18006)}

\end{thebibliography}

\end{document}